\documentclass[12pt]{amsart}
\usepackage[english]{babel}
\usepackage[
hypertexnames=false, colorlinks, citecolor=red, linkcolor=red]{hyperref}
\hypersetup{bookmarksdepth=3}

\usepackage{amssymb,graphicx,textcomp,bbm}
\usepackage{mathtools}
\usepackage{
mathrsfs}

	\normalbaselines						  %

\newcommand{\assign}{:=}
\newcommand{\mathd}{\mathrm{d}}
\newcommand{\nocomma}{}
\newcommand{\tmop}[1]{\ensuremath{\operatorname{#1}}}

\newtheorem{theorem}{Theorem}[section]
\newtheorem{prop}[theorem]{Proposition}
\newtheorem{lm}[theorem]{Lemma}
\theoremstyle{definition}
\newtheorem{definition}[theorem]{Definition}
\newtheorem*{definition*}{Definition}
\theoremstyle{remark}
\newtheorem{rem}[theorem]{Remark}
\newtheorem*{rem*}{Remark}

\numberwithin{equation}{section}

\newcommand{\1}{\mathbf 1}

\newcommand{\ci}[1]{_{ {}_{\scriptstyle #1}}}
\newcommand{\ti}[1]{_{\scriptstyle \text{\rm #1}}}

\newcommand{\cD}{\mathcal{D}}
\newcommand{\cX}{\mathcal{X}}

\newcommand{\cF}{\mathcal{F}}
\newcommand{\cS}{\mathcal{S}}

\newcommand{\R}{\mathbb{R}}
\newcommand{\E}{\mathbb{E}}

\newcommand{\Z}{\mathbb{Z}}

\newcommand{\dd}{\mathrm{d}}

\newcommand{\fdot}{\,\cdot\,}

\newcommand{\sd}{{\scriptstyle\Delta}}
\newcommand{\rk}{\operatorname{rk}}
\newcommand{\Ran}{\operatorname{Ran}}
\newcommand{\ch}{\operatorname{ch}}

\renewcommand{\labelenumi}{(\roman{enumi})}

\newcounter{vremennyj}

\newcommand\cond[1]{\setcounter{vremennyj}{\theenumi}\setcounter{enumi}{#1}\labelenumi\setcounter{enumi}{\thevremennyj}}

\begin{document}

\title[Lower square function estimates]{On the failure of lower square function estimates in the non-homogeneous
weighted setting}

\author{K.~Domelevo}

\author{P.~Ivanisvili}

\author{S.~Petermichl}

\author{S.~Treil}
\address{S.~Treil: Department of Mathematics \\ Brown University \\ Providence, RI 02912 \\ USA}
\email{treil@math.brown.edu}

\author{A.~Volberg}
\address{A.~Volberg: Department of Mathematics \\ Michigan State University \\ East Lansing, MI, 48824 \\ USA}
\email{volberg@math.msu.edu}

\thanks{ Komla Domelevo and  Stefanie Petermichl  were supported by ERC project CHRiSHarMa no. DLV-682402. Sergei Treil was  supported by the NSF grants DMS-1600139. Alexander Volberg was supported by NSF grant DMS-1600065. The paper was written while the authors were in residence at the Mathematical Sciences Research Institute in Berkeley, California, during the Spring semester of 2017, 
 supported by the National Science Foundation under Grant No. 1440140. }

\begin{abstract}
  We show that the classical $A_{\infty}$ condition is not sufficient for a
  lower square function estimate in the non-homogeneous weighted $L^2$ space.
  We also show that under the martingale $A_2$ condition, an estimate holds
  true, but the optimal power of the characteristic jumps from $1 / 2$ to $1$
  even when considering the classical $A_2$ characteristic. This is in a sharp
  contrast to known estimates in the dyadic homogeneous setting as well as the
  recent positive results in this direction on the discrete time
  non-homogeneous martingale transforms. Last, we give a sharp $A_{\infty}$
  estimate for the $n$-adic homogeneous case, growing with $n$. \ 
\end{abstract}

{\maketitle}

\setcounter{tocdepth}{1}
\tableofcontents

\section{Introduction}

\

It is a classical result that the Haar system on the real line is an unconditional basis in the weighted space $ L^2(w)= L^2 (\R,w)$
if and only if the weight $w$ satisfies the dyadic Muckenhoupt $A_2$ condition. 
This is
equivalent to boundedness of the predictable $\pm 1$ multiplier on the
martingale difference sequences with underlying homogeneous dyadic filtration.
This generalizes to martingale difference spaces in homogeneous 
filtrations. These results
were proved in \cite{GunWhe}, where also Littlewood--Paley estimates were considered. It
has been known for some time that the optimal unconditional basis constants
are the first power of the $A_2$ characteristic of the weight. Through
averaging, it follows that the square function has no worse upper bounds, so, again, at most the first power of the $A_2$ characteristic of the weight. 

Concerning the lower estimate of the square function,
it is known that the square function for the standard dyadic filtration on $\R$ satisfies better lower estimates, namely,  with a
square root on the characteristic instead of linear --- the upper and lower estimates  estimates are both
optimal for the homogeneous filtration, see \cite{HuTrVo_SquareFn_2000}, \cite{PetPott_2000}.  In fact, even the weaker $A_{\infty}$ characteristic
is sufficient for this lower estimate (for the standard dyadic filtration on $\R$), also with square root bounds; this was proved in \cite{Wilson1989} using the earlier results from  \cite{ChWiWo1985}.

It was a general understanding that in the homogeneous case one should have the same lower bounds as in the case of the standard dyadic filtration on $\R$, but surprisingly, it was not proven before for our ``real'' square function. The result from \cite{Wilson1989} gives the desired estimates for a bigger square function, but the statement for our ``real'' square function (which is the only one that works in the non-homogeneous case) for a homogeneous filtration is proved (to the best of our knowledge) only in the present paper.


The  sharp results on the  estimates of unconditional basis constants for arbitrary underlying
Radon measure and any discrete in time atomic filtration was proved more
recently in \cite{ThTrVo-A2} and then later in \cite{La2017-SimpleA2} by a different method. The constants remain in a linear
dependence with the martingale $A_2$ characteristic, exactly as in the homogeneous situation.

\

In this paper, we discuss the upper and lower estimates of the square
function in this (arbitrary filtration) setting. It is remarkable that the better lower estimates
seen in the homogeneous setting fail --- indeed the $A_{\infty}$ bound does not
hold true at all --- in other words, the $A_{\infty}$ condition is not
sufficient for a lower square function bound. This is even so when using the
most restrictive way of defining $A_{\infty}$. Under the martingale $A_2$
condition, we obtain a lower estimate, but we will see that it is twice the
power of that in the homogeneous case. The failure of the lower estimates
motivate us to look closely at the $n$-adic homogeneous case --- one expects a
growth with $n$. Indeed, we show that the lower square function estimate in
this setting still holds under the $A_{\infty}$ assumption, but with a growth
$O (n)$.

\

To see the blow ups we claim, we construct weights, in $A_2$ or $A_{\infty}$
respectively, via their martingales based on a filtration where each interval
has at most two children, but of possibly very disbalanced measures.

To see the $A_{\infty}$ lower estimate via the true square function in the
$n$-adic setting, we make use of a Bellman functional taking a distribution
function as its variable. This idea stems from \cite{TV-Entropy2016} --- but here is an additional
difficulty, similar to that of estimating Haar shifts with Bellman functions.

\section{Setup and motivations}

\subsection{Filtered atomic spaces}

Let $(\cX, \cF, \nu)$  be a $\sigma$-finite measure space with an atomic filtration, meaning that there exist an increasing sequence of $\sigma$-algebras $\cF_n$, $n\in\mathbb N$ or $n\in\Z$, such that for each $n$ there exists a countable collection $\cD_n$ of sets of finite positive measure (called \emph{atoms}) such that $A \in \mathcal{F}_n$ is a union of atoms of $\mathcal{D}_n$. 


We will denote $I\in{\mathcal D}_n$ the atoms of ${\mathcal D}_n$, and
denote by ${\mathcal D}$ the collection of all atoms, i.e. $\mathcal{D}= \cup_{n}\cD_n$.
We allow a set $I$ to belong to several generations $\mathcal{D}_n$, so formally
an atom $I\in\mathcal{D}_n$ is a pair $(I,n)$. When there is no confusion, we will omit the ``time" $n$
and write simply $I$ instead of $(I,n)$; otherwise when it is necessary to
refer to the time $n$, we will use the symbol $\tmop{rk} (I)$, such that if $I$ denotes the atom $(I,n)$ then
$\tmop{rk} (I) \assign n$. Also the inclusion $I \subset J$ for atoms should be understood as inclusion
for the sets together with the inequality $\tmop{rk} (I) \geqslant \tmop{rk} (J)$.
However the union (intersection) of atoms will simply denote the union (intersection)
of the corresponding sets regardless of the time component.
For $I\in\mathcal{D}_n$ we denote by $\tmop{ch}(I)$ the set of children of $I$, that is
the atoms of $\mathcal{D}_{n+1}$ that are direct descendants of $I$ :
$\tmop{ch}(I) \assign \{ I'\in\mathcal{D}_{n+1} ; I'\subset I  \}$.

A typical example is the standard dyadic filtration in $\R^d$ with $\nu$ being an arbitrary Radon measure $\nu$; of course, we need to ignore all cubes $Q\in\cD$ with $\nu(Q)=0$.

To avoid nonessential technical details, in this paper we assume that $\nu$ is a probability measure, and the filtration is indexed by $n\in\Z_+$. We also assume that $\cD_0=\{\cX\}$, and each $\cD_n$ is a finite collection (i.e.~that every  atom has finitely many children).  

Since our main results are counterexamples, by providing them in more restrictive settings we get a formally stronger result than in the more general settings. As for the positive estimates, they can be extended to the general case using standard approximation reasoning, so we do not lose anything.

Without loss of generality we can assume that $\cX$ is the unit interval $[0,1]$, the measure $\nu$ is the standard Lebesgue measure, and that the atoms are intervals. 
We assume that the $\sigma$-algebra $\cF$ is generated by $\sigma$-algebras $\cF_n$, so more precisely, $\nu$ is the restriction of the Lebesgue measure on $\cF$. 


Measures of intervals are denoted by $| I |\assign \nu(I)$. 
For any interval $I \in\mathcal{D}$, we define
\begin{equation}
\label{eq: definition average}
\langle f \rangle\ci I = |I|^{-1} \int_I f \mathd \nu 
\end{equation}
and
\[
\mathbbm{E}\ci I f = \langle f \rangle\ci I \1\ci I .
\]
For any interval $I \in\mathcal{D}$, the martingale difference operator $\Delta\ci I$ is defined by
\[ \Delta\ci I f = \sum_{I' \in \tmop{ch} (I)} \mathbbm{E}\ci {I'} f -\mathbbm{E}\ci I
   f. \]
Notice that the atom $I\in \mathcal{D}_n$ has only one child (i.e. $\tmop{ch}(I)=\{I\}$) if and only if
the corresponding martingale difference operator is trivial (i.e. $ \Delta\ci I = 0$).

With this in mind, setting
\[ \mathbbm{E}_n f = \sum_{I \in \mathcal{D}_n} \mathbbm{E}\ci I f =\mathbbm{E}
   (f | \mathcal{F}_n),\]
	
we define the martingale difference operator $\Delta_n$ for any $n > 0$ as
\[
\Delta_n f =\mathbbm{E}_n f -\mathbbm{E}_{n - 1} f
	=	\sum_{I \in \mathcal{D}: \tmop{rk} (I) = n - 1}   \Delta\ci I f
\]
together with $\Delta_0 f =\mathbbm{E}_0 f =\langle f\rangle\ci\cX \1$. In the sum above
the contributions of the trivial martingale operators is automatically omitted.

For $I \in \mathcal{D}$ denote by $D\ci I$ the
martingale difference space, the image of the operator $\Delta\ci I$, so $D\ci I =
\Delta\ci I L^2$ and similarly $D_n = \Delta_n L^2$.  Note, that the subspaces $D_n$, $n\ge0$ form an orthogonal basis in $L^2 =L^2(\cX, \cF, \nu)$, and the same holds for the family $D\ci I$, $I\in\cD$ together with the subspace $D_0$ (consisting of constants).  

\subsection{Bases of martingale difference spaces and the Muckenhoupt \texorpdfstring{$A_2$}{A<sub>2} condition}
In the setting described above the following statements are equivalent (with equivalent constants in statements \cond3--\cond6) as a
consequence of the general theory of bases, cf. \cite{ThTrVo-A2}.
\begin{enumerate}
  \item The system of subspaces $\{ D\ci I : I \in \mathcal{D} \nocomma, D\ci I \neq
  \{ 0 \} \}\cup \{D_0\}$ is an unconditional basis in $L^2 (w)$.
  
  \item The system of subspaces $\{ D_n : 0 \leqslant n < \infty \nocomma, D_n
  \neq \{ 0 \} \}$ is an unconditional basis in $L^2 (w)$.
  
  \item The predictable martingale multipliers $T_{\sigma}$ $T_{\sigma} f =
  \sum_{I \in \mathcal{D}} \sigma\ci I \Delta\ci I f$, with $\sigma = \{
  \sigma\ci I \}\ci {I \in \mathcal{D}}$, $\sigma\ci I \in \{ 0, 1 \}$ (or equivalently $\sigma\ci I\in\{-1,1\}$),    are uniformly in $\sigma$
  bounded in $L^2 (w)$.
  
  \item The predictable martingale multipliers $T_{\sigma}$ with $\sigma = \{
  \sigma\ci I \}\ci {I \in \mathcal{D}}$, $| \sigma\ci I | \leqslant 1$ are uniformly
  in $\sigma$ bounded in $L^2 (w)$.
  
  \item The martingale multipliers $T_{\tau}$ with $\tau = \{ \tau_n \}_{n \in
  \mathbbm{N}}$, $\tau_n \in \{ 0, 1 \}$ (or, equivalently $\tau_n\in\{-1,1\}$),  
  \[
  T_{\tau} f = \sum_{k \in \mathbbm{N}} \tau_k \Delta_k f
  \] 
  are uniformly in $\tau$ bounded in $L^2(w)$.
  
  \item The martingale multipliers $T_{\tau}$ with $\tau = \{ \tau_n \}_{n \in
  \mathbbm{N}}$, $| \tau_n | \leqslant 1$ are uniformly in $\tau$ bounded in
  $L^2 (w)$.
\end{enumerate}

It has been known for some time that the statements \cond3--\cond6 hold if and only if the weight $w$ satisfies the martingale Muckenhoupt $A_2$ condition, see Definition \ref{d:A_2-dy} below: for the standard dyadic filtration in $\R^N$ we can refer the reader to \cite{GundyWheeden1973}, and for general martingales the result was proved in \cite{Bonami_Mart_A2_1979}. Later it was proved that the constants in the statements \cond4--\cond6 are estimated by the first power of the $A_2$ characteristic (i.e. $\lesssim   [w]\ci {2, \mathcal{D}}$):  for the standard dyadic filtration in $\R$ (and so in $\R^N$)it was proved in \cite{Wittwer2000}; for the general non-homogeneous filtration  it was established in \cite{ThTrVo-A2} and soon after by a different method in \cite{La2017-SimpleA2}.

By taking the average over all $\sigma\ci I\in\{-1,1\}$ one can see that for a weigh satisfying the martingale $A_2$ condition, the quantity $\|Sf\|\ci{L^2(w)}$ is equivalent in the sense of two sided estimates to the norm $\|f\|\ci{L^2(w)}$, see the details in Section \ref{s: trivial estimates}. It can be easily obtained from the estimate $\|T_\sigma\|\ci{L^2(w)\to L^2(w)} \lesssim [w]\ci{2,\cD}$ that   
\begin{align*}
[w]\ci{2,\cD}^{-1} \|f\|\ci{L^2(w)} \lesssim \|Sf\|\ci{L^2(w)} \lesssim [w]\ci{2,\cD} \|f\|\ci{L^2(w)} \qquad \forall f\in L^2(w), 
\end{align*}
see again Section \ref{s: trivial estimates} for details. The upper bound $\|Sf\|\ci{L^2(w)} \lesssim [w]\ci{2,\cD} \|f\|\ci{L^2(w)}$ is known to be sharp, but the lower bound $\|f\|\ci{L^2(w)} \lesssim [w]\ci{2,\cD} \|Sf\|\ci{L^2(w)}$, as we discussed above in the introduction, can be improved in the homogeneous case. The investigation of the lower bound in the non-homogeneous situation was the main motivation for this paper.

\subsection{Different \texorpdfstring{$A_2$}{A2} and \texorpdfstring{$A_{\infty}$}{A<sub>infty} conditions}

Since our underlying filtration can be non-homogeneous, we have to be very careful
about the definitions of the classes of weights we will use, as they are no
longer necessarily comparable. In all definitions we consider integrable $w$.
Also the notation $\langle \cdot \rangle_I$ below denotes the average operator as defined in \eqref{eq: definition average}.

\begin{definition}
\label{d:A_2-dy}
  We say that a weight $w$ satisfies the martingale $A_2$ condition and write $w \in A_{2}^ {\mathcal{D}}$ if
  \[ [w]\ci {2, \mathcal{D}} := \sup_{I \in \mathcal{D}} \langle w \rangle\ci I
     \langle w^{- 1} \rangle\ci I < \infty . \]
\end{definition}

\begin{definition}
  We say that a weight $w$ satisfies the classical $A_2$ condition and write $w \in A_{2}^{\tmop{cl}}$ if
\[ 
  [w]_{2}^{\tmop{cl}} = \sup_{I \subseteq [0, 1]} \langle w \rangle\ci I
     \langle w^{- 1} \rangle\ci I < \infty , 
\]
  where the supremum runs over all intervals $I \subset [0, 1]$.
\end{definition}

\begin{definition}
For an interval $I$ define the \emph{localized} maximal function $M\ci I$, 
\[
M\ci I f(x) :=  \1\ci I (x) \sup_{J \subseteq I:\, x \in J } | \langle
     f \rangle\ci J |,  
\]
where the supremum runs over all intervals $J \subset I$ containing $x$. 
 
For an interval $I\in\cD$ define also the \emph{martingale} localized maximal function $M^\cD\ci I$,  
  \[ M^{\mathcal{D}}_I f (x) = \1\ci I (x) \sup_{J \in\cD(I) :\, x \in J }
    \left| \langle f \rangle\ci J \right| \]

\end{definition}

\begin{definition}
  We say that a weight $w$ satisfies the classical $A_\infty$ condition and write $w \in A_{\infty}^{ \tmop{cl}}$ if
\[ 
[w]_{\infty, \tmop{cl}} = \sup_{I \subseteq [0, 1]} \frac{\langle
     M_I w \rangle\ci I}{\langle w \rangle\ci I} < \infty  .
 \]
where  $M_I f$ is the localized classical maximal function defined above. 

\end{definition}

\begin{definition}
  We say that a weight $w$ satisfies the \emph{semiclassical} $A_\infty$ condition and write  $w \in A_{\infty}^{\tmop{scl}}$ if
\[ 
  [w]_{\infty, \tmop{scl}} = \sup_{I \in \mathcal{D}} \frac{\langle
     M\ci I w \rangle\ci I}{\langle w \rangle\ci I} < \infty , 
\]
  where again $M_I f$ is the classical maximal function localized
  to $I \in \mathcal{D}$.
\end{definition}

\begin{definition}
  We say that $w \in A_{\infty}^{\mathcal{D}}$ if
\[ 
[w]\ci {\infty, \mathcal{D}} = \sup_{I \in \mathcal{D}} \frac{\langle
     M^{\mathcal{D}}_I w \rangle\ci I}{\langle w \rangle\ci I} < \infty , 
\]
 where $M_I^\cD f$ is the martingale maximal function localized
  to $I \in \mathcal{D}$. 
\end{definition}

We need the following well-known fact. 
\begin{prop}
For any atomic filtration
\begin{equation}
[w]\ci {\infty, \mathcal{D}}\le 4 [w]\ci {2, \mathcal{D}}
\label{eq: A2 versus Ainfinity}
\end{equation}
\end{prop}
For a simple (but probably not the first) proof see \cite[Lemma 4.1]{NPetTrV2017}; there it was stated for the standard dyadic filtration on $\R^d$, but the same proof without any changes works for any atomic filtration. 

It is a theorem of \cite{ThTrVo-A2}  and \cite{La2017-SimpleA2} that the $A_{2}^{\mathcal{D}}$ characteristic is
sufficient, indeed that the constants above are bounded by a multiple of
$[w]\ci {2, \mathcal{D}}$. It is well known that the $A_{2}^{\mathcal{D}}$
condition is necessary and that the linear dependence in \eqref{eq: A2 versus Ainfinity} is optimal among all
estimates of the form $\Phi ([w]\ci {2, \mathcal{D}})$, which is already seen in
the case of dyadic filtration with underlying Lebesgue measure.

\

\section{Main results}

For $f\in L^1(\cX)$ the martingale
square function is defined by 
\[ S h := \left( \sum_I (\Delta\ci I h)^2 \right)^{1
   / 2} . \]
There are variations in the literature that are not equivalent when the
measures are non-homogeneous. Ours is the most natural definition from probability
theory, and the only one that works in the non-homogeneous case. For example, for our square the quantity $\|Sf\|_p$ is always equivalent to the norm $\|f\|_p$, $1<p<\infty$, (with constants depending on $p$); for other accepted definitions of a square function the equivalence of the norms is true only for homogeneous filtrations, but fails in the non-homogeneous case for $p\ne 2$. 

In the paper the expression  $A \lesssim B$ means there exists a universal constant $c$, independent of
the important quantities, such as function, weight, measure and filtration, so
that $A \leqslant c B$. If the constant depends on some parameters, say $a$ and $b$, we will write  $A\underset{a,b}{\lesssim} B$.

The theorem below is presented just for the sake of completeness. Estimate \eqref{UpperBd-01} can be easily obtained from known results, see Section \ref{s: trivial estimates} below. A bit stronger estimate \eqref{UpperBd} can be obtained from the upper bound (Theorem \ref{t:UpperBd} below) via Proposition \ref{p:lwr-uprBd-S}.
\begin{theorem}
  \label{thm: A2 lower}Given the interval $[0, 1]$ and any discrete time
  atomic filtration and any measure, then there holds
\begin{align}
\label{UpperBd}
 \| f \|\ci {L^2 (w)}  & \lesssim [w]^{1 / 2}\ci {2, \mathcal{D}} [w]\ci {\infty,
     \mathcal{D}}^{1 / 2} \| S f \|\ci {L^2 (w)}
   \\ \label{UpperBd-01}
& \le 2 [w]\ci {2, \mathcal{D}} \| S f \|\ci {L^2 (w)} . 
\end{align}
\end{theorem}

Here are our main theorems
\begin{theorem}
  \label{thm: A2 classical} The exponent $1$ of $[w]\ci {2, \mathcal{D}}$ in \eqref{UpperBd-01} is optimal. 
Namely, given $A\ge 1$ one can find a weight $w$ defined on the interval $[0, 1]$ satisfying the classical $A_2$ conditions, such that   $[w]_{2, \tmop{cl}}=A$ and a non-homogeneous dyadic filtration $\cD$ such that for some $f\in L^2(w)$
\[
\| f \|\ci{L^2(w)} \gtrsim A \| Sf \|\ci{L^2(w)} = [w]\ci{2, \tmop{cl}}\| Sf \|\ci{L^2(w)};
\]
recall that the implied constant here is an absolute one. 
\end{theorem}

Since $[w]\ci {2, \mathcal{D}} \leqslant [w]\ci {2, \tmop{cl}}$ this indeed means that the estimate $\| f \|\ci {L^2 (w)} \lesssim [w]\ci {2, \mathcal{D}} \| S
f \|\ci {L^2 (w)}$ in Theorem \ref{thm: A2 lower} is sharp.

\begin{theorem}
  \label{thm: Ainfinity lower}
  Assumption  $w\in A_\infty^{\tmop{cl}}$ is not sufficient for an estimate 
  \[
  \|f\|\ci{L^2(w)} \le C([w]_{\infty,\tmop{cl}})\|Sf\|\ci{L^2(w)}.
  \]  Namely, one can find a weight $w$ on the interval $[0, 1]$ satisfying the classical $A_\infty$ condition and a non-homogeneous dyadic filtration for which there exists a sequence of functions $f_n\in L^2(w)$ with
  \[
  \|Sf_n\|\ci{L^2(w)} =1, \qquad \|f_n\|\ci{L^2(w)}\to \infty\quad \text{as }n\to \infty. 
  \]
  \end{theorem}

Since $[w]\ci {\infty, \tmop{cl}} \geqslant [w]\ci {\infty, \tmop{scl}} \geqslant
[w]\ci {\infty, \mathcal{D}}$ this means in particular that no definition of
$A\ci {\infty}$ is sufficient for a lower square function estimate in the
non-homogeneous case.

The following theorem can be obtained combining results from \cite{ChWiWo1985} and \cite{Wilson1989}, but here we present a direct proof. 

Recall that the $n$-adic filtration is the atomic filtration where each atom has exactly $n$ children of equal measure. 
\begin{theorem}
  \label{thm: Ainfinity n adic}For the $n$-adic filtration
  \[ 
  \| f \|\ci {L^2 (w)} \lesssim n [w]^{1 / 2}\ci {\infty, \tmop{scl}} \| S f
     \|\ci {L^2 (w)} . 
     \]
\end{theorem}

\begin{rem}
The above theorem holds for an arbitrary \emph{homogeneous} filtration, i.e.~for a filtration such that for a certain constant $C\ti h>0$,
\[
\forall I\in\mathcal{D},\  \forall I'\in\tmop{ch}(I),\   | I | \le C\ti h  |I'|. 
\]
Then it can be seen from the proof that 
\[
\| f \|\ci {L^2 (w)} \underset{C\ti h}{\lesssim } [w]^{1 / 2}\ci {\infty, \tmop{scl}} \| S f
     \|\ci {L^2 (w)} .
\]

In particular, there holds $\| f \|\ci {L^2 (w)} \underset{n}{\lesssim} [w]^{1 / 2}\ci {\infty,
\mathcal{D}} \| S f \|\ci {L^2 (w)}$ with additional growth in $n$.

\end{rem}

The following result is probably well-known, see for example \cite{LaceyLi2016} for the version for a continuous square function. 
We present it just for the completeness, and we will just outline the proof of \eqref{eq: upper bound harder} in Section \ref{S: upper bound for the square function} and the proof of \eqref{eq: upper bound easy} in Section \ref{s: trivial estimates}.

\begin{theorem}
\label{t:UpperBd}
For an arbitrary atomic filtration and a weight $w\in A_2^\cD$
\begin{align}
\label{eq: upper bound harder}
\| S f\|\ci{L^2(w)} & \lesssim [w]\ci{2,\cD}^{1/2} [w^{-1}]\ci{\infty,\cD}^{1/2} \|f\|\ci{L^2(w)}
\\ 
\label{eq: upper bound easy}
 & \le 2 [w]\ci {2, \mathcal{D}} \|  f \|\ci {L^2 (w)}
\end{align}
\end{theorem}

\section{Reduction of lower bound to an embedding theorem}

It is more convenient to treat the square function $S$ as a linear operator, by paying the price of treating it as an operator to the space of vector-valued functions. 

Namely, define $\vec{S} : L^2_0 \rightarrow L^2 (\ell^2)$ as
\[ 
\vec S h  
= \{
   \Delta\ci I h \}\ci{I\in\cD}.
\]
Here we treat the sequence $\{\Delta\ci I h \}\ci{I\in\cD}$ as an element of the $\ell^2$-valued space $L^2(\ell^2)$, i.e. we associate with this sequence the function $\vec S h$ of two variables, $x\in \Omega$, $k\in \mathbb{N}$, 
\begin{align*}
\vec S h (x, k) = \Delta\ci I h (x), \qquad \text{where } I\in \cD \text{ is such that } \rk (I) = k. 
\end{align*}
Since for all $x\in\Omega$
\[
| S h(x) | = \|\vec S h (x, \fdot) \|\ci{\ell^2},
\]
we conclude that 
\begin{align}
\label{S-vecS-02}
\| Sh\|\ci{L^2(w)} =  \|\vec S h  \|\ci{L^2(w;\,\ell^2)}  :=\biggl( \int_\Omega \|\vec S h(x, \fdot) \|\ci{\ell^2}^2 w(x) \dd x \biggr)^{1/2}.
\end{align}

So the estimates for the square function $S$ are equivalent (with the same constants) to the corresponding estimates for the vector-valued square function $\vec S$. 

\subsection{Trivial estimates}
\label{s: trivial estimates}

Let $T_\sigma$, $\sigma=\{\sigma\ci I\}\ci{I\in\cD}$, $\sigma\ci I\in\{-1,1\}$ be a martingale multiplier, 
\begin{align*}
T_\sigma f = \sum_{I\in\cD} \sigma\ci I \Delta\ci I f .
\end{align*}
Taking the average $\E_\sigma$ over all possible choices of $\sigma\ci I\in\{-1,1\}$ (i.e.~formally taking $\sigma\ci I$ to be independent random variables taking values $\pm1$ with probability $1/2$), we conclude that for almost all $x$
\begin{align*}
\E_\sigma \left( |T_\sigma f(x)|^2\right) = \left( S f(x)\right)^2.
\end{align*}
Therefore, for any weight $w$ and any $f\in L^2(w)$
\begin{align*}
\inf_\sigma \| T_\sigma f\|\ci{L^2(w)} \le \|Sf \|\ci{L^2(w)} \le \sup_\sigma \| T_\sigma f\|\ci{L^2(w)}.
\end{align*}
Thus, denoting by $M(w):= \sup_\sigma \|T_\sigma\|\ci{L^2(w)\to L^2(w)}$ we can see that
\begin{align*}
M(w)^{-1}\|  f\|\ci{L^2(w)} \le \|Sf \|\ci{L^2(w)} \le M(w) \|  f\|\ci{L^2(w)}.
\end{align*}
It is well known that for $w\in A\ci{2}^{\cD}$
\begin{align*}
\| T_\sigma \|\ci{L^2(w)\to L^2(w)} \lesssim [w]\ci{2,\cD}; 
\end{align*}
for the classical dyadic filtration on $\R$ this result was first proved in \cite{Wittwer2000}, and many different proofs are known now for homogeneous filtrations. For the non-homogeneous case it was proved in \cite{ThTrVo-A2} and then independently and by a different and easier method in \cite{La2017-SimpleA2}. 

In fact, using the sparse domination technique from \cite{La2017-SimpleA2} one can show that for any atomic filtration one can write the following (stronger) $A_2$--$A_\infty$ estimate
\begin{align}
\| T_\sigma \|\ci{L^2(w)\to L^2(w)} \lesssim [w]\ci{2,\cD}^{1/2} 
\left( [w]\ci{\infty,\cD}^{1/2} + [w^{-1}]\ci{\infty,\cD}^{1/2} \right).
\label{eq: Tsigma stronger}
\end{align}

Another trivial observation is that a lower bound for $Sf$ in $L^2(w)$ can be reduced to the upper bound in $L^2(w^{-1})$:
\begin{prop}
\label{p:lwr-uprBd-S}
Let $w>0$ a.e. Then
\begin{align}
\label{lwr-uprBd-S}
\|f\|\ci{L^2(w)} \le \|S\|\ci{L^2(w^{-1})\to L^2(w^{-1})} \|S f\|\ci{L^2(w)}
\end{align}
\end{prop}

\begin{proof}
By \eqref{S-vecS-02} estimates for $S$ are reduced to estimating its ``linearized'' vector-valued version $\vec S$. Namely, it is sufficient to estimate the norm in $L^2(w)$ of the \emph{canonical} left inverse $\vec{S}^{- 1, \tmop{left}}$ of $\vec S$, 
\[
\vec{S}^{- 1, \tmop{left}}: \Ran \vec{S}\to L^2;
\]
note that since $\vec S$ is clearly an injective map, the operator $\vec{S}^{- 1, \tmop{left}}$ is well defined. Note also that there are no weights in the definition of $\vec{S}^{- 1, \tmop{left}}$.

The operator $\vec S : L^2\to L^2(\ell^2)$ (in the non-weighted situation) is an isometry, so
\begin{align*}
\vec{S}^{- 1, \tmop{left}} = \vec{S}^* \Bigm| \Ran \vec{S}. 
\end{align*}
Therefore
\begin{align}
\label{S^-1-S^*}
\| \vec{S}^{- 1, \tmop{left}} \|\ci{L^2(w;\,\ell^2)\to L^2(w)} \le \| \vec{S}^{*} \|\ci{L^2(w;\,\ell^2)\to L^2(w)}. 
\end{align}
But for an operator $\vec{S}^* : L^2(w;\,\ell^2)\to L^2(w)$ its adjoint with respect to the standard non-weighted duality is the operator $\vec{S} : L^2(w^{-1})\to L^2(w^{-1};\,\ell^2)$, so 
\begin{align*}
\| \vec{S}^{- 1, \tmop{left}} \|\ci{L^2(w;\,\ell^2)\to L^2(w)} \le \|\vec{S}\|\ci{L^2(w^{-1})\to L^2(w^{-1};\,\ell^2)}, 
\end{align*}
which immediately gives \eqref{lwr-uprBd-S}. 

In the above reasoning we skipped a trivial technical detail, namely that $\vec{S} L^2 \ne \vec{S} L^2(w)$ and we have to be a bit careful. However, it all can be fixed by a standard approximation reasoning. For example, for a finite $\cF\subset\cD$ we can define the square function $S\ci\cF$,  
\[
S\ci \cF h =\Biggl(\sum_{I\in\cF} |\Delta\ci I h|^2 \biggr)^{1/2}. 
\] 
Then for the vector version $\vec{S}\ci\cF$ we do not have a problem with ranges, so the above reasoning gives us the estimate \eqref{lwr-uprBd-S} with $S\ci\cF$ instead of $S$. Taking the supremum over all finite $\cF\subset\cD$ we get \eqref{lwr-uprBd-S}. 
\end{proof}

\subsection{A sharper way to write the lower bound for the square function}
\label{s:better lower bound}


Analyzing the proof of Proposition \ref{p:lwr-uprBd-S}, we can see where one could lose sharpness  of the estimate (and in some cases we indeed do lose it): we estimate the norm of the operator $\vec{S}^*$ between weighted spaces, while we need to estimate only the norm of its restriction, which could be smaller. 

We wish to
find a more convenient equivalent form of the inequality
\begin{equation}
  \| h \|\ci {L^2 (w)} \leqslant C \| S h \|\ci{L^2 (w)} \label{eq: original lower}
\end{equation}
that gives us the same constant in the estimate.

Denoting $h\ci I:=\Delta\ci I h $ the above inequality reads, with the same constant $C$ as above,
\begin{align}
\label{eq: original lower 01}
\biggl\| \sum_{I\in\tilde\cD} h\ci I \biggr\|_{L^2 (w)}
= \biggl\| \sum_{I\in\cD} h\ci I \biggr\|_{L^2 (w)}
	\leqslant C \biggl( \sum_{I\in\cD} \| h\ci I  \|^2\ci {L^2 (w)} \biggl)^{1 / 2}
	= C \biggl( \sum_{I\in\tilde\cD} \| h\ci I  \|^2\ci {L^2 (w)} \biggl)^{1 / 2} ,  
\end{align}
where we noted in the first and last sum $\tilde\cD = \{I\in\cD : h_I\neq 0\}$.

The standard approximation reasoning implies that it is sufficient to check the above inequality only for finite sums, so we do not have to worry about convergence.  

The sequence $\{h\ci I\}\ci{I\in\cD}$ is a sequence of martingale differences: this simply means that each $h\ci I =\Delta\ci I h $ for some $h$, or, equivalently, that $h\ci I$ is supported on $I$, $\int h\ci I \dd x=0$  and $h\ci I$ is
constant on all $I' \in \tmop{ch} (I)$.

The above inequality \eqref{eq: original lower 01} holds for all finite sequences $\{h\ci I\}\ci{I\in\cD}$ of martingale differences if and only the estimate 
\begin{align}
\label{eq: original lower 02}
\biggl\| \sum_{I\in\tilde\cD}  x\ci I h\ci I \biggr\|_{L^2 (w)}
\leqslant C \biggl( \sum_{I\in\tilde\cD} x^2\ci I   \| h\ci I \|^2\ci {L^2 (w)} \biggr)^{1 / 2} 
\end{align}
holds for all (finite) collections of martingale differences $h\ci I$ and real numbers $x\ci I$, $I\in\tilde\cD$. The fact that \eqref{eq: original lower 02} implies \eqref{eq: original lower 01} is trivial; on the other hand denoting $x\ci I h\ci I$ in \eqref{eq: original lower 02} by $h\ci I$ we can see that \eqref{eq: original lower 01} implies  \eqref{eq: original lower 02}. 

It looks like we just made the estimate \eqref{eq: original lower 01} more complicated, but this allows us to reduce the problem to a simple ``embedding theorem''. 
 
Namely,   for a fixed sequence $\{ h\ci I \}\ci {I\in\cD}$  of martingale differences let us define the \emph{reconstruction} operator
\[ 
R : \ell^2 =\ell^2(\tilde\cD) \rightarrow L^2 ,  \qquad R x = \sum_{I\in\tilde\cD} x\ci I h\ci I , \quad \text{where } x= \{ x\ci I \}\ci {I\in\tilde\cD}\,.
\]
With respect to the unweighted pairing, its adjoint is the operator 
\begin{align}
\label{R^*}
R^* : L^2 \rightarrow \ell^2,  \qquad R^* f = \{ (f, h\ci I)\ci {L^2} \}\ci {I\in\tilde\cD}  . 
\end{align}
Define $\gamma=\{\gamma\ci I\}\ci{I\in\tilde\cD}= \Bigl\{\|h \ci I\|\ci{L^2(w)}^2\Bigr\}_{I\in\tilde\cD}$\,, and the norm in the weighted space $\ell^2(\gamma)$ is given by 
\begin{align*}
\|x\|\ci{ \ell^2(\gamma) }^2  = \sum_{I\in\tilde\cD} x\ci I^2 \gamma\ci I. 
\end{align*}
The estimate \eqref{eq: original lower 02} can be rewritten as 
\begin{align*}
\|R x \|\ci{L^2(w)} \le C \|x\|\ci{\ell^2(\gamma)}.
\end{align*}

But that is equivalent to the weighted estimate 
\begin{align}
\label{eq: R lower 02}
\|R \|\ci{\ell^2(\gamma)\to L^2(w)} \le C 
\end{align}

For the operator $R:\ell^2 (\gamma)\to L^2(w)$ its adjoint with respect to the standard non-weighted duality is the operator 
\begin{align*}
R^* : L^2(w^{-1}) \to \ell^2(\gamma^{-1})
\end{align*}
where $\gamma^{-1}= \{\gamma\ci I^{-1}\}\ci{I\in\tilde\cD}$, and $R^*$ is the adjoint of the operator $R$ in the non-weighted situation ($R:\ell^2\to L^2$, $R^*:L^2\to\ell^2$) given by \eqref{R^*}.

The inequality \eqref{eq: R lower 02} (and so \eqref{eq: original lower 02}) rewritten for the adjoint operator thus becomes
\[ 
\sum_{I\in\tilde\cD} \frac{(f, h\ci I)\ci {L^2}^2}{\gamma\ci I} \leqslant C^2 \int^1_0 | f |^2
   w^{- 1} \]
and writing $f = g w$ we can restate it as
\begin{align}
\label{Weighted Embedding Haar 01}
\sum_{I\in\tilde\cD} \frac{(g, h\ci I)^2\ci {L^2 (w)}}{\gamma\ci I} \leqslant C^2 \int^1_0 | g |^2
   w. 
\end{align}

Let us simplify the estimate \eqref{Weighted Embedding Haar 01} a bit more. 
Consider the weighted Haar functions $h\ci I^w$, 
\[ 
h\ci I^w = h\ci I - d\ci I \1\ci I ,  
\]
where $d_I$ is the unique constant such that  $h^w_I  \bot \1\ci I$ in $L^2
(w)$. Thanks to orthogonality   we have by Pythagorean theorem the estimate
$\| h^w\ci I  \|\ci {L^2 (w)} \leqslant \| h\ci I \|\ci {L^2 (w)}$. Notice further that with this choice
of Haar functions, we have $\tilde \cD = \{I\in\cD; \tmop{ch}(I)\neq I\}$.
In particular, if $\cD$ is the usual dyadic or $n$-adic filtration, then $\tilde\cD=\cD$.
This is the situation we will consider in the counterexamples built in the next sections.

In order to estimate the  sum in \eqref{Weighted Embedding Haar 01}, it suffices to estimate the terms
\[ 
\sum_{I\in\tilde\cD} \frac{(g, h^w\ci I)\ci {L^2 (w)}^2}{\gamma\ci I} \qquad \tmop{and}\qquad \sum_{I\in\tilde\cD}
   \frac{d^2\ci I  (g, \1\ci I)\ci {L^2 (w)}^2}{\gamma\ci I} . 
 \]
The first sum is easily estimated by the Pythagorean theorem:
\begin{align}
\notag
\sum_{I\in\tilde\cD} \frac{(g, h^w\ci I)\ci {L^2 (w)}^2}{\gamma\ci I} & = \sum_{I\in\tilde\cD} \frac{(g, h^w\ci I)\ci {L^2 (w)}^2}{\| h\ci I\|\ci{L^2(w)}^2}
\\ 
\label{Triv sum}
& \le \sum_{I\in\tilde\cD} \frac{(g, h^w\ci I)\ci {L^2 (w)}^2}{\| h\ci I^w\|\ci{L^2(w)}^2} \le \|f\|\ci{L^2(w)}. 
\end{align}

The second sum can be rewritten as 
\begin{align*}
\sum_{I \in \tilde\cD} \frac{d^2\ci I \langle gw \rangle^2\ci I | I
   |^2}{\gamma\ci I}
\end{align*}
and 
by the martingale Carleson Embedding theorem, it suffices to check its bounds on functions $g=\1\ci J$, $J\in\tilde\cD$. 

Namely, this sum is bounded by  $C_1^2 \|f\|\ci{L^2(w)}^2$ if and only if  for all $J\in\tilde\cD$
\begin{align}
\label{eq: modified testing 01}
\frac{1}{| J |} \sum_{I \in\tilde\cD(J)} \frac{d^2\ci I \langle w \rangle^2\ci I | I
   |^2}{\gamma\ci I} \leqslant C^2_2 \langle w \rangle\ci J\,.   
\end{align}
Combining this estimate with \eqref{Triv sum} and using the triangle inequality for the $\ell^2$ norm, we get that \eqref{eq: modified testing 01} holds if and only if
\begin{equation}
\label{eq: modified testing}
  \frac{1}{| J |} \sum\ci {I \in\tilde\cD( J)} \frac{(w, h\ci I)\ci {L^2}^2}{\gamma\ci I}
  \leqslant C^2_3 \langle w \rangle\ci J \qquad \forall J\in\tilde\cD.   
\end{equation}
Moreover, we can see that the best constants in inequalities \eqref{eq:
original lower}, \eqref{eq: modified testing 01} and \eqref{eq: modified testing} are equivalent.

\section{Counterexample for the \texorpdfstring{$A_2$}{A2} lower bound.}

In this section, we will prove Theorem \ref{thm: A2 classical}; note that it is sufficient to prove this theorem for sufficiently large $A$. 

We will first construct a non-homogeneous dyadic filtration on $I_0=[0,1]$ and a weight $w$ with  $[w]\ci{2, \tmop{cl}} =A$ such that for the best constant $C_3$ in \eqref{eq: modified testing} we have for this filtration $C_3 \gtrsim  [w]\ci{2, \tmop{cl}}$. More precisely, we will prove the estimate 
\begin{align}
\label{eq: testing counterexample}
\sum_{I \in\cD(I_0)} \frac{(w, h\ci I)\ci {L^2}^2}{\| h\ci I\|\ci{L^2(w)}^2}
  \gtrsim A^2\langle w \rangle\ci {I_0} \,.
\end{align}

Then later in Section \ref{ss: classical A_2} we will show that the weight $w$ we constructed belongs to the classical $A_2$ class, and that  $[w]\ci{2, \tmop{cl}}\asymp [w]\ci{2, \cD}$, which completely proves Theorem \ref{thm: A2 classical}. 

Note, that since our filtration is dyadic, all martingale difference subspaces $\Delta\ci I L^2$ are one-dimensional, so the Haar functions $h\ci I$ are uniquely defined up to a factor. Due to homogeneity of each term in \eqref{eq: testing counterexample} a choice of the factor does not matter. 


\subsection{Preliminary computations and idea of the proof}
For an interval $I\in\cD$ let $I_+$ and $I_-$ be its children, and let 
\[
\alpha_\pm^I := |I_\pm|/|I|. 
\]
The corresponding Haar function $h\ci I$ is given (up to a constant factor) by 
\[
h\ci I = \alpha_-^I \1\ci{I_+} - \alpha_+^I \1\ci{I_-}. 
\]
Then 
\[
(w, h\ci I)\ci{L^2} = \alpha_+^I\alpha_-^I \left(\langle w
   \rangle\ci{I_+} - \langle w \rangle\ci{I_-}\right) |I|, 
\]
and 
\[
\|h\ci I \|\ci{L^2(w)}^2 = \alpha_+^I\alpha_-^I \left( \alpha^I_- \langle w
   \rangle\ci {I_+} + \alpha^I_+ \langle w \rangle\ci {I_-} \right) |I|, 
\]
so the left hand side in \eqref{eq: testing counterexample} is given by 
\begin{align}
\label{eq: testing sum}
\sum_{I \in \cD(I_0)} \frac{\alpha^I_- \alpha^I_+ \left(\langle w
  \rangle_{I_+} - \langle w \rangle_{I_-}\right)^2}{\alpha^I_- \langle w
  \rangle_{I_+} + \alpha^I_+ \langle w \rangle_{I_-}} | I | .
\end{align}

\subsubsection{Idea of the construction}
Assume we have for a term in the sum \eqref{eq: testing sum}  $\alpha_-\ll \alpha_+$  (and in particular $\alpha_-^I\le 0.1$, so $\alpha_+^I\ge 0.9$). Assume also for this term $\alpha^I_- \langle w \rangle_{I_+} \approx  \alpha^I_+ \langle w \rangle_{I_-}$ so $\langle w \rangle_{I_+} - \langle w \rangle_{I_-}\gtrsim \langle w \rangle_{I}$,
and let also $ \langle w \rangle\ci I  | I |  \gtrsim \langle w \rangle\ci {I_0}  |I_0|$. Then term we have
\begin{align*}
\frac{\alpha^I_- \alpha^I_+ \left(\langle w
  \rangle_{I_+} - \langle w \rangle_{I_-}\right)^2}{\alpha^I_- \langle w
  \rangle_{I_+} + \alpha^I_+ \langle w \rangle_{I_-}} | I | \gtrsim \langle w \rangle\ci I  | I |
  \gtrsim \langle w \rangle\ci {I_0}  |I_0| .
\end{align*}
If we are able to find as many as $A^2$ such intervals, we will prove \eqref{eq: testing counterexample}, and therefore also Theorem \ref{thm: A2 classical}.

So let us construct a (non-homogeneous) dyadic filtration $\cD$ and a weight $w\in A_2$ such that $[w]\ci{2, \tmop{cl}} =A$ such that we have sufficient number of terms as we described above. 

In the construction we first show that $[w]\ci{2, \cD} =A$, and later prove that the classical $A_2$ characteristic remains the same. 

\subsubsection{A random walk representation}
To construct a weight we will use its \emph{martingale representation} i.e.~get the weight from a random walk in the domain $\Omega\ci A\subset \R^2$, 
\begin{align*}
\Omega\ci A:= \{ (u,v)\in \R^2 : \, 1\le uv\le A\}. 
\end{align*}
Namely, suppose for each $I\in\cD$ we have a point $X\ci I = (u\ci I, v\ci I) \in \Omega\ci A$, and the points $X\ci I$ satisfy a (non-homogeneous) \emph{martingale dynamics}, 
\begin{align}
\label{eq: martingale dynamics}
X\ci I = \alpha_+^I X\ci{I_+}+\alpha_-^I X\ci{I_-};
\end{align}
here recall $\alpha_\pm^I = |I_\pm|/|I|$. 

This collection of points $X\ci I$ can be interpreted as as a non-homogeneous random walk in $\Omega\ci A$, where we move from a point $X\ci I$ to points $X\ci{I_\pm}$ with probabilities $\alpha_\pm^I$ respectively. 

In our example the walk will be stopped after $n$ steps on the lower boundary $uv=1$ of $\Omega\ci A$, meaning that for all $I\in \ch^{k} I_0$, $k> n$ we have
\begin{align*}
u\ci I v\ci I = 1. 
\end{align*}
\begin{rem*}
Note that when the walk hits the lower boundary $uv=1$ of $\Omega\ci A$ it must stay there; it is immediate corollary of the martingale dynamics \eqref{eq: martingale dynamics} and the requirement that one must stay above the hyperbola $uv=1$. 
\end{rem*}

Such a walk immediately gives us a weight $w\in A_2^\cD$. Namely, take the level $N$ where the walk is stopped on the hyperbola $uv=1$, and define
\begin{align*}
w:= \sum_{I\in\ch^N I_0} u\ci I \1\ci I \,.
\end{align*}

The martingale dynamics \eqref{eq: martingale dynamics}
together with the fact that $u\ci I v\ci I = 1$ for all $I\in \ch^N(I_0)$ imply that for any $I\in\cD$
\begin{align}
\label{eq:w from uv}
\langle w\rangle\ci I = u\ci I, \qquad \langle w^{-1}\rangle\ci I = v\ci I\,.
\end{align}
Since $X\ci I\in \Omega\ci A$, identities \eqref{eq:w from uv} mean that $[w]\ci{2,\cD} \le A$; if we, for example start the walk at a point on the upper hyperbola $uv=A$, then trivially $[w]\ci{2,\cD} = A$. 

\subsection{The construction}
Let us construct the non-homogeneous dyadic filtration and the corresponding random walk in $\Omega\ci A$, which gives us the weight $w$ as follows. 

\subsubsection{Setting up the random walk}
We restrict our attention to the one dimensional dyadic setting. Let $I_0=[0,1]$. The dyadic filtration $\cD(I_0)$ is such that each $I\in\cD$ has exactly 2 children, $I_+$ and $I_-$, with equal Lebesgue measure $\lambda(I_-)=\lambda(I_+)=\lambda(I)/2$. However, with respect to the non homogeneous measure $\nu$, we have $\nu(I\pm)\assign |I_\pm |  \assign \alpha_\pm^I |I|$, and we will be choosing  the probabilities $\alpha_\pm^I$ in order to completely define the dyadic lattice.

For easier bookkeeping let $I_+$ always be on the right, and let $|I_+|\ge|I_-|$.  

We start from the interval $I_0=[0,1]$, and pick a point $X_0=X\ci{I_0}=(u_0,v_0)$ on the upper hyperbola $uv=Q_0=A$.  We will then construct the random walk in such a way, that at each interval $I$ anything interesting can happen only on its right part $I_+$; on the left part $I_-$ the walk stops on the lower hyperbola $uv=1$. Because we are stopped on the lower hyperbola, it does not matter how we continue the filtration $\cD$ on $I_-$; we can, for example continue it as the standard dyadic filtration. 

So, we start from the interval $I_0$, and anything interesting will happen only on its right part  
$(I_0)_+ =: I_1$, because the walk will stop on $(I_0)_- =: I_1^\star$. We then split the interesting interval $I_1$ into two parts $I_2:= (I_1)_+$ and $I_2^\star:= (I_1)_-$, so again on $I_2^\star$ the walk stops, and so on\ldots 

So, we will only need to keep track of what is going on on intervals $I_k$, $I_k^\star$, $k\ge 1$
\begin{align*}
I_{k+1}:= (I_k)_+, \quad I_{k+1}^\star := (I_k)_-, \qquad k\ge 0. 
\end{align*}
Denoting for simplification of notation the corresponding probabilities $\alpha_\pm^I$ by $\alpha_k$ and $\alpha_k^\star$,  we write 
\begin{align*}
|I_{k+1}| = \alpha_k|I_{k}|, \quad | I_{k+1}^\star | = \alpha_k^\star |I_{k}|, \qquad k\ge 0
\end{align*}
(clearly $\alpha_k+\alpha_k^*=1$); the values of $\alpha_k$, $\alpha_k^\star$ will be chosen later.

The points $X_k =(u_k, v_k)$, $X_k^\star =(u_k^\star, v_k^\star)$ of our walk must satisfy the martingale dynamics \eqref{eq: martingale dynamics}, which in our notation can be rewritten as 
\begin{align}
\label{eq: martingale dynamics 01}
X_k = \alpha_k X_{k+1} + \alpha_k^\star X_{k+1}^\star. 
\end{align}

Schematically, the random walk we need to track can be presented in the picture below. 

\

\begin{center}
  \includegraphics[width=10cm]{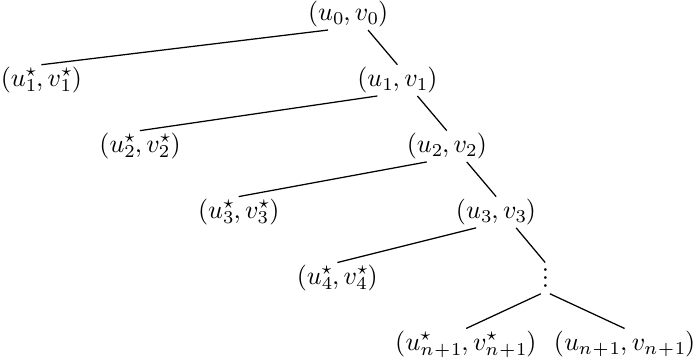}
\end{center}

\subsubsection{Inductive construction}
We start from a point $X_0=(u_0, v_0)$, $u_0v_0=Q_0:=A$, and construct the the walk by induction. Suppose we constructed the points $X_1, X_2, \ldots, X_k$, and $X_1^\star, X_2^\star, \ldots, X_k^\star$, and let $Q_k:= u_k v_k$. We will continue our iterations as long as $Q_k\ge Q_0/2$; if $Q_k< Q_0/2$  we stop the walk by moving from the point $X_k$ to the both points being on the lower hyperbola $uv=1$. 

If $Q_k \ge Q_0/2$ we set 
\begin{align}
\label{eq: alpha_k}
\alpha_k^\star = 1/Q_k, \qquad \alpha_k = 1- \alpha_k^\star .
\end{align}
The point $X_{k+1}^\star$ is defined as the point of intersection of the tangent line to the hyperbola $uv=Q_k$ at the point $X_k=(u_k,v_k)$ and the lower hyperbola $uv=1$. The computations show
\begin{align*}
u_{k+1}^{\star} = \left(1 - \sqrt{1 - 1/Q_k}\right) u_k, \qquad v_{k+1}^{\star} =
   \left(1 + \sqrt{1 - 1/Q_k}\right)v_k ;
\end{align*}
probably the easiest way to compute is to do first the computations for the case $u_k =v_k =Q_k^{1/2}$ and then do the rescaling $u\mapsto \lambda u $, $v\mapsto \lambda^{-1} v$   for an appropriate $\lambda$. 

It follows from the martingale dynamics \eqref{eq: martingale dynamics 01} that 
\begin{align*}
u_{k+1} & = \left( 1 + \frac{\alpha_k^{\star}}{\alpha_k} \sqrt{1 - 1 /
   Q_k} \right)  u_k , \qquad & v_{k+1} &=  \left( 1 -
   \frac{\alpha_k^{\star}}{\alpha_k} \sqrt{1 - 1 / Q_k} \right) v_k, \\
 & = \left( 1 + \alpha_k^{\star} {\alpha_k}^{-1/2} \right)  u_k,  & & = \left( 1 - \alpha_k^{\star} {\alpha_k}^{-1/2} \right)  v_k.
\end{align*}

The figure below shows an example of a dyadic martingale as above with $X_k=(u_k,v_k)$ with $0\leqslant k \leqslant 4$, $X_k^\star=(u_k^\star,v_k^\star)$, with $1\leqslant k\leqslant 3$. Only $X_0$, $X_1$ and $X_0^\star$ are labelled. The two hyperbolas are $u v = 1$ and $u v = Q_0=A$. All the points lie in the domain $\Omega_A$.
\begin{center}
  \includegraphics[width=10cm]{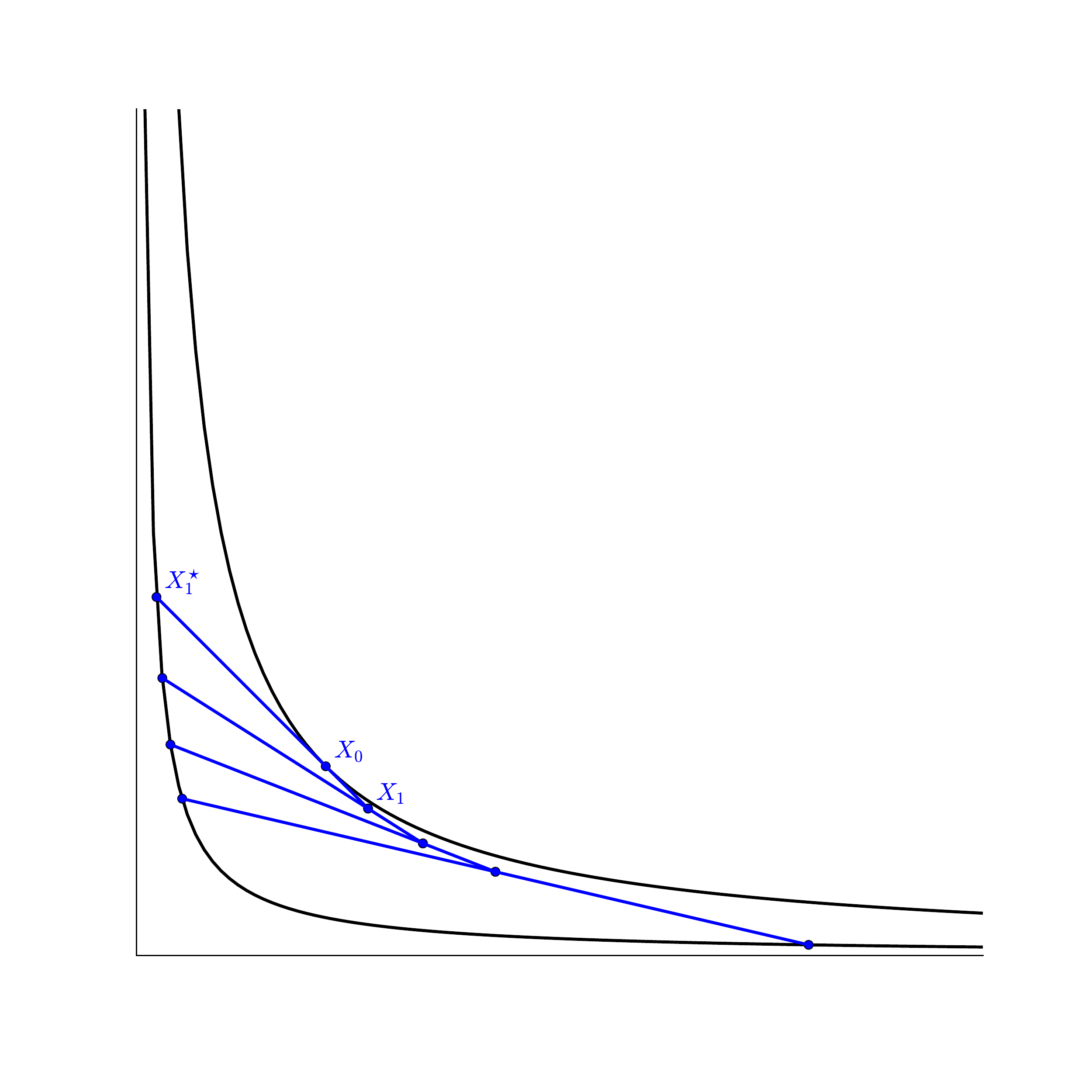}
\end{center}

\subsubsection{The estimates}
Let us now write some estimates. Let us assume that $Q_0=A\ge 4$, so $Q_k \ge A/2=Q_0/2 \ge 2$. Then 
\begin{align*}
u_{k+1} - u_{k+1}^\star & \ge u_k - u_{k+1}^\star = u_k\sqrt{1-1/Q_k} \ge u_k/\sqrt2, \\
  \alpha_k^\star u_{k+1} + \alpha_k u_{k+1}^\star  & = \left[\alpha_k^\star (1 + \alpha_k^\star \alpha_k^{-1/2}) + \alpha_k (1-\alpha_k^{1/2}) \right] u_k 
\\ 
& \le   \left[\alpha_k^\star (1 + \alpha_k^\star \alpha_k^{-1/2}) + \alpha_k^\star \alpha_k \right] u_k \lesssim \alpha_k^\star u_k.
\end{align*}

Combining the above estimates together we get that 
\begin{align}
\label{eq:gtrsim u_k I_k}
\frac{\alpha_k\alpha_k^\star (u_{k+1} - u_{k+1}^\star)^2}{\alpha_k^\star u_{k+1} + \alpha_k u_{k+1}^\star  } |I_k| \gtrsim u_k |I_k| 
\end{align}

Using formulas for $u_{k+1}$ and $u_{k+1}^\star$ we get that 
\begin{align}
\notag
Q_{k+1} & = \left( 1 + \alpha_k^{\star} {\alpha_k}^{-1/2} \right)\left( 1 - \alpha_k^{\star} {\alpha_k}^{-1/2} \right) Q_k \\
\notag
&= \left( 1 - Q_k^{-2} (1-1/Q_k)^{-1}\right) Q_k \\
\label{eq:Q_{k+1}}
& \ge \left( 1 - 2 Q_k^{-2} \right) Q_k \ge \left( 1 - 8 Q_0^{-2} \right) Q_k \, .
\end{align}
Finally, since $u_{k+1} = (1+\alpha_k^\star \alpha_k^{-1/2}) u_k$ we get 
\begin{align}
\notag
u_{k+1} |I_{k+1}| & = (1-\alpha_k^\star) (1+\alpha_k^\star \alpha_k^{-1/2}) u_k |I_k| \\
\notag
& \ge \left(1- (\alpha_k^\star)^2 \right) u_k |I_k|  = \left(1- 1/Q_k^2 \right) u_k |I_k| \\
\label{eq:u_{k+1}I_{k+1}}
& \ge \left(1- 4/Q_0^2 \right) u_k |I_k| .
\end{align}

The estimate \eqref{eq:Q_{k+1}} implies that 
\begin{align*}
Q_k \ge \left( 1 - 8 Q_0^{-2} \right)^k Q_0 , 
\end{align*}
so for $n\gtrsim Q_0^2$ steps we will have $Q_k\ge Q_0/2$, $k\le n$. Finally, it follows from \eqref{eq:u_{k+1}I_{k+1}} that
\begin{align*}
u_k |I_k| \ge \left(1- 4/Q_0^2 \right)^k u_0 |I_0|,
\end{align*}
therefore $u_k |I_k| \ge \frac12 u_0 |I_0|$ for $k\le n$. 
From \eqref{eq:gtrsim u_k I_k} we get that for $k\le n$
\begin{align*}
\frac{\alpha_k\alpha_k^\star (u_{k+1} - u_{k+1}^\star)^2}{\alpha_k^\star u_{k+1} + \alpha_k u_{k+1}^\star  } |I_k| \gtrsim u_0 |I_0| 
\end{align*}
\subsubsection{Finishing the random walk} 
\label{ss:finishing walk}
First of all let us note that in our construction not only the points $X_k$, $X_k^\star$, but the whole interval $[X_k,X_k^\star]$  are in the domain $\Omega\ci A$. That will be needed in proving that the weight $w$ we constructed satisfies the classical $A_2$ condition and that $[w]\ci{2,\cD} = [w]\ci{2,\tmop{cl}}$. 

Note also that the following follows immediately from the construction:
\begin{enumerate}
\item The sequence $u_k$ is increasing, the sequence $v_k$ is decreasing. 
\item The sequence $Q_k$ is decreasing. 
\item The slopes of intervals $[X_k^\star, X_k]$ are negative and increasing (i.e.~have decreasing absolute values). 
\end{enumerate}

In our construction we made $n$ steps while $Q_k\ge Q_0/2$. Now we need to stop the process by moving from $X_n$ to the points $X_{n+1}$, $X_{n+1}^\star$ on the lower hyperbola $uv=1$. Note that we can easily do it preserving the above properties \cond1--\cond3; recall that we have a  choice of transition probabilities $\alpha_n$, $\alpha_n^\star$.

\subsection{Why the constructed weight belongs to  classical \texorpdfstring{$A_2$}{A2}}
\label{ss: classical A_2}

It is of independent interest to observe that even classical $A_{2}^
{\tmop{cl}}$, containing many more intervals as competitors, is not sufficient
for a square root bound. We will show that the example above indeed belongs to the classical $A_2$ and that $[w]\ci{2,\cD} = [w]\ci{2,\tmop{cl}}$. 

The following argument is borrowed from \cite{IOSVZ}. Let $X:I_0\to \R^2$ be a vector-valued function, $X(t) = (w(t), w(t)^{-1})$. 

Consider the trajectory
\[ 
\gamma (t) \assign 
\left\langle X
   \right\rangle_{[t, 1]} , \qquad t \in I_0=[0, 1] . \]
Notice that $\gamma (0) = (w_0, v_0)$ is the starting point. Let $\beta_k$ be
the left endpoint of the interval $I_k$, then
\begin{align}
\label{eq:endpoints of I_k}
\gamma (\beta_k) = (1-\beta_k) X_k , \qquad X_k =(u_k, v_k).  
\end{align}
Since the weight is constant on the interval $I_{k + 1} \setminus I_k$ we see that on this interval 
the trajectory of $\gamma(t)$ in the $uv$ plane  is exactly  the line segment joining the points $X_{k}$ and $X_{k+1}$ (note that this segment is the part of the interval $[X_k^\star, X_k]$).

Indeed, since both $w$ and $w^{-1}$ are constant on $I_{k+1}\setminus I_k$, both $u$ and $v$ coordinates of $\gamma(t)$ have a form 
\begin{align*}
\frac{a+bt}{1-t} = \frac{a+b}{1-t} - b, 
\end{align*}
so both coordinates are affine functions of the variable $s=1/(1-t)$. Therefore the trajectory indeed lies on a line segment. The monotonicity of the change of variables  $s=1/(1-t)$ together with \eqref{eq:endpoints of I_k} insure that this segment is exactly $[X_k, X_{k+1}]$.

Clearly the trajectory of $\gamma (t)$ is
convex (increasing slopes, see \cond3 in Section \ref{ss:finishing walk} above),  piecewise linear, and  it belongs to the domain 
\[ 
\Omega_{A} \assign \{(u, v) \in \mathbb{R}^2 : 1 \leq uv \leq A \} . 
\]
The line segments at the endpoints of the curve $\gamma$ if extended to the
line liees below the graph $uv = A$ (here we should agree that on the final
interval $I_n$ we concatenated the weight along the line segment not
intersecting the previous line segments and the boundary $uv = A$).

Take arbitrary $1 \geq b > a \geq 0$. Since
\[ \gamma (a) = \frac{1 - b}{1 - a} \cdot \gamma (b) + \frac{b - a}{1 - a}
   \cdot \langle X\rangle_{[a, b]}, \]
it follows from a simple geometry that $\left\langle X \right\rangle_{[a, b]} \in \Omega_{Q_0}$.
The figure below illustrates the equation above. Notice that the segment $[ \langle X\rangle_{[a, b]}, \gamma (a)]$
lies below the convex curve $\gamma(t)$ and below its tangent at $t=0$. This ensures that  $\langle X\rangle_{[a, b]}$
belongs to $\Omega_A$.

\

\begin{center}
  \includegraphics[width=10cm]{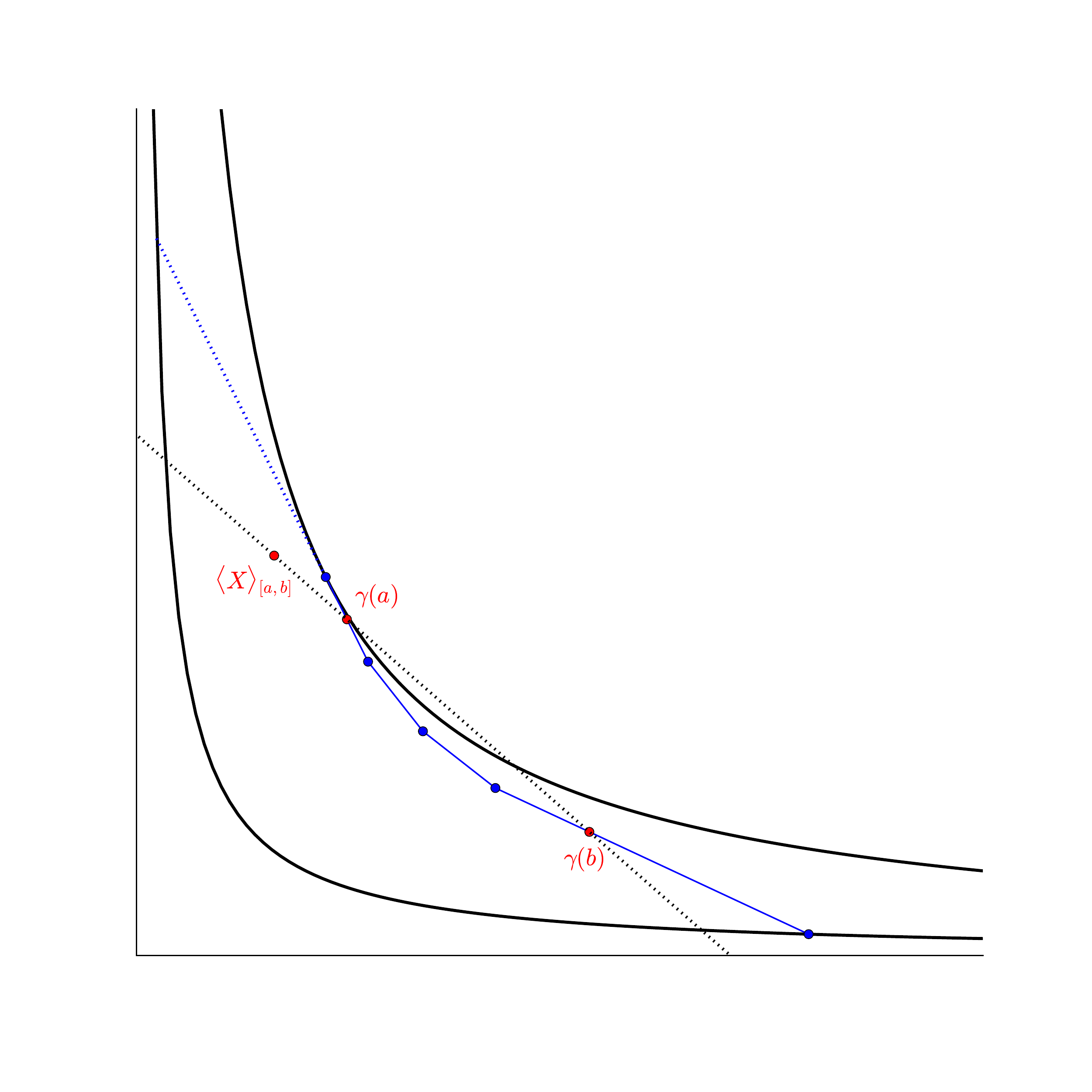}
\end{center}

\section{No bounds in terms of \texorpdfstring{$A_{\infty}$}{A<sub>infty}}

In this section we prove Theorem \ref{thm: Ainfinity lower}. We show that in
the non-homogeneous setting, if $[w]_{\infty,{\tmop{cl}}} < \infty$ then we
can choose a filtration so that the sum
\[ 
\frac{1}{| J |} \sum_{I \subseteq J} \frac{\alpha^I_+ \alpha^I_-  (\langle
   w \rangle\ci {I_+} - \langle w \rangle\ci {I_-})^2}{\alpha^I_+  \langle w
   \rangle\ci {I_-} + \alpha^I_-  \langle w \rangle\ci {I_+}} |I| 
\]
can be very large (so no bound in terms of $A\ci {\infty, \tmop{cl}}$
characteristics can be obtained).

Indeed, Take $w (x) = x$ on $[0, 1]$. It is not difficult to check that
$[w]\ci {{\infty},{\tmop{cl}}}$ is finite. Let $\varepsilon > 0$ be a
sufficiently small number (we will specify it later). We will construct the
filtratrion as follows (parent $\to$ children)

\begin{align*}
   I_0 &:= [0, 1] ; \qquad  &I_0^- & \assign [0, \varepsilon], \hspace{1em}
  &I_0^+ &:= [\varepsilon, 1] ;\\
   I_1 &\assign I_0^+ ;  &I_1^- &\assign [\varepsilon, 2
  \varepsilon], \hspace{1em} & I_1^+ &\assign [2 \varepsilon, 1] ;
  \\
   \ldots\\
   I_{k - 1} &:= [(k - 1) \varepsilon, 1] ; &I_{k - 1}^- &:= [(k -
  1) \varepsilon, k \varepsilon], \hspace{1em} &I_{k - 1}^+ &\assign [k
  \varepsilon, 1]
\end{align*}

Then
\begin{align*}
   \langle w \rangle_{I_{k - 1}^-} &= \frac{\varepsilon (2 k - 1)}{2} ; \qquad &
   \langle w \rangle_{I_{k - 1}^+} & = \frac{1 + \varepsilon k}{2} ;\\
  \alpha_{I_{k - 1}^-} &\assign \frac{\varepsilon}{1 - \varepsilon (k - 1)}
  ; 
  & \alpha_{I_{k - 1}^+} &\assign \frac{1 - \varepsilon k}{1 - \varepsilon (k
  - 1)} .
\end{align*}

Let's say we make $N$ steps. Then
\begin{align*}
   \sum_{k = 1}^N \frac{\alpha_{I_{k - 1}^-} \alpha_{I_{k - 1}^+} 
  (\langle w \rangle_{I_{k - 1}^-} - \langle w \rangle_{I_{k -
  1}^+})^2}{\alpha_{I_{k - 1}^-} \langle w \rangle_{I_{k - 1}^+} +
  \alpha_{I_{k - 1}^+} \langle w \rangle_{I_{k - 1}^-}} |I_{k - 1} | 
  = \frac{1}{2}  \sum_{k = 1}^N \frac{(1 - \varepsilon k)  (1 -
  \varepsilon (k - 1))^2}{(1 + \varepsilon k) + (1 - \varepsilon k)  (2 k -
  1)} .
\end{align*}

Choose $\varepsilon = \frac{1}{N}$. Then
\begin{eqnarray*}
  \frac{1}{2}  \sum_{k = 1}^N \frac{(1 - \varepsilon k)  (1 - \varepsilon (k -
  1))^2}{(1 + \varepsilon k) + (1 - \varepsilon k)  (2 k - 1)} & \geqslant &
  \frac{1}{8}  \sum_{k = 1}^N \frac{(1 - k / N)^3}{k}\\
  & \geqslant & \frac{1}{8}  \sum_{k = 1}^N \frac{1 - 3 k / N}{k}\\
  & \geqslant & \frac{1}{8}  (\ln (N - 1) - 3)
\end{eqnarray*}
and it becomes very large as $N \to \infty$.

\

\section{Estimate in terms of martingale \texorpdfstring{$A_{\infty}^\cD$}{A<sub>infty} for 
homogeneous filtrations}

In this section we prove Theorem \ref{thm:
Ainfinity n adic}. 

Since everything scales correctly, we can assume without loss of generality that the starting interval $I_0$ of our filtration is $I_0=[0,1]$. 
 
Let \ $\mathcal{D}=\cD(I_0)$
denote all $n$-adic intervals $I\subset I_0$. 

\subsection{Bellman functional and its properties}
For a non-negative function $w$ on an interval $I$ let $N=N_I^w$ be its \emph{normalized} distribution function, 
\begin{align}
\label{eq:NormDistrFn}
N_I^w (t):= |I|^{-1}\left|\left\{ x\in I : w (x)>t\right\}\right|, \qquad t\ge 0, 
\end{align}
Trivially the normalized distribution function $N_I^w$ satisfies the \emph{martingale dynamics}, namely, if $I_k$ are the children of $I$, then 
\begin{align*}
N_I^w= \sum_k \alpha_k N_{I_k}^w, \qquad \text{where } \alpha_k = |I_k|/|I|. 
\end{align*}

On the set of distribution functions consider the Bellman functional
\[ B (N) = \int^{\infty}_0 \psi (N (t)) \mathd t \]
with $\psi (s) = s - s \ln (s)$. 

We will need the following well-known fact, see \cite[Theorem IV.6.7]{Bennett-Sharpley_1988}.
\begin{lm}
\label{l:LlogL}
Let $w$ be a non-negative function on $I_0=[0,1]$ and let $N=N_{I_0}^w$ be its distribution function. Then $\|M_{I_0} w\|\ci{L^1}$ and $B(N)$ are equivalent in the sense of two-sided estimates (with some absolute constants). 
\end{lm}

Let $N=N_0$ and $N_1$ be two distribution functions, and let $\sd N:=N_1 -N$. We want to compute the second derivative of the function $\theta \mapsto B(N+\theta\sd N)$. 

Let $N_\theta :=N + \theta\sd N$, and let 
\[
u_\theta := \int_0^\infty N_\theta(t) dt. 
\]
If we think of the function $N_\theta$ as of the distribution function of a function $w_\theta$ on, say, $[0,1]$, then $u_\theta$ is the average of the function $w_\theta$. 
Also, denote
\begin{align}
\label{delta-u}
\sd u := u_1-u_0 = \int_0^\infty \sd N(t) \dd t.
\end{align}
Then we calculate
\[ 
\frac{\mathd^2}{\mathd \theta^2} B (N_{\theta}) = \frac{\mathd^2}{\mathd
   \theta^2} \int^{\infty}_0 \psi (N_{\theta} (t)) \mathd t = -
   \int^{\infty}_0 \frac{(\sd N (t))^2}{N_{\theta} (t)} \mathd t. 
\]
Using the Cauchy--Schwartz inequality we get, see \cite[Lemma 5.1]{TV-Entropy2016}, that
\[ 
- \frac{\mathd^2}{\mathd \theta^2} B (N_{\theta}) \geqslant \frac{\left(
   \int^{\infty}_0 \sd N (t) \mathd t \right)^2}{\int^{\infty}_0 N_{\theta}
   (t) \mathd t} = \frac{| \sd u |^2}{u_{\theta}} .  
\]
Then using the Taylor's formula we get, see \cite[Corollary 5.2]{TV-Entropy2016} 
\begin{lm}
\label{l:two term Bellman inequality}
Let $N_1$, $N_2$ and $N$ be distribution functions such that $N=(N_1+N_2)/2$ and $N=N(N_{1,2})<\infty$. Let $\sd N=N_1-N$ and $\sd u$ is defined by \eqref{delta-u}.  Then 
\begin{align}
\label{eq:two term Bellman inequality}
B(N)-\frac{B(N_1)+B(N_2)}{2} \ge \frac12\cdot \frac{(\sd u)^2}{u},  
\end{align}
where, recall $u=\int_0^\infty N(t) \dd t$. 
\end{lm}

Using this lemma one can easily get the result for the dyadic filtration. To get it for the $n$-adic filtration some extra work is needed. 

\begin{definition}
\label{df:Haar}
Recall that a Haar function on an interval $I\in\cD$  is a function $h=h\ci I$ supported on $I$, constant on children of $I$ and such that $\int_I h\ci I \dd x=0$. 

A Haar function $h\ci I$ is called elementary if it is non-zero on at most 2 children of $I$. Thus any elementary Haar function $h\ci I$ can be represented as $h\ci{I}= c\ci I \left( \1\ci{I_{k_1}} - \1\ci{I_{k_2}}\right)$, $I_{k_1}, I_{k_2}\in\ch I$. 
\end{definition}

\begin{lm}
\label{l: elementary Haar decomposition}
Let $\cD$ be an $n$ adic filtration. 
Any Haar function $h$ on an interval $I\in \cD$ can be represented as a sum of at most $n$ elementary Haar functions $h_k$, and moreover
\begin{align}
\label{eq: sum |h_k|}
|h|=\sum_k |h_k|
\end{align}
\end{lm}

\begin{proof}
We prove it using induction in $n$. The case $n=2$ is trivial. 

Suppose the lemma is proved for $n-1$. Let $I_k$ be the children of $I$. We write $h$ as 
\[
h=\sum_{k=1}^n \eta_k \1\ci{I_k} . 
\]
Since $\int_h \dd x =0$ there exist $k_1$, $k_2$ such that $\eta_{k_1}>0$, $\eta_{k_2}<0$. 

For ${\mu}_1
\assign \min (| \eta_{k_1} |, | \eta_{k_2} |)$ define 
\[
h_1:= \mu_1 \left( \1\ci{I_{k_1}} -\1\ci{I_{k_2}}\right), \qquad h^1 := h-h_1.  
\]
Clearly, $h_1$ is an elementary Haar function, $h$ is a Haar function and 
\begin{align}
\label{eq: |h_1|+|h^1|}
|h| = |h_1| + |h^1|. 
\end{align}
Note, that $h^1$ is supported on at most $n-1$ intervals. Applying the induction hypothesis we get the decomposition $h=\sum_k h_k$. Identity \eqref{eq: sum |h_k|} follows  from \eqref{eq: |h_1|+|h^1|}. 
\end{proof}

\subsection{Proof of Theorem \ref{thm: Ainfinity n adic}}
We need to estimate the left hand side of \eqref{eq: modified testing}, i.e.~the sum
\begin{equation}
\label{eq: modified testing 02}
  \frac{1}{| I_0 |} \sum\ci {I \in\cD( I_0)} \frac{(w, h\ci I)\ci {L^2}^2}{\|h\ci I\|\ci{L^2(w)}^2}  
\end{equation}
Recall that for an interval $I\in\cD\ci I$, we note $N\ci I^w$ the distribution function \eqref{eq:NormDistrFn}. We want to show that 
\begin{align}
\label{eq:main Bellman inequality}
|I| B(N\ci I) - \sum_{I_k\in\ch I} |I_k| B(N\ci{I_k}) \ge 
\frac{2}{ n^2} \frac{(w, h_{I})_{L^2}^2}{\| h_{I} \|^2_{L^2
   (w)}}
\end{align}
Then summing over all $I\in\cD(I_0)$ and taking into account that $B(N\ci I)\ge 0$ we get that
\begin{align*}
\sum_{I\in\cD(I_0)} \frac{(w, h_{I})_{L^2}^2}{\| h_{I} \|^2_{L^2
   (w)}} \le \frac{n^2}{2} B(N\ci{I_0})  \lesssim n^2 \| M\ci{I_0} w\|\ci{L^1(I_0)} ;
\end{align*}
the last inequality here follows from Lemma \ref{l:LlogL}. By the definition of $A_\infty$
\[
\| M\ci{I_0} w\|\ci{L^1(I_0)} \le [w]\ci{\infty, \tmop{cl}} \langle w\rangle\ci{I_0} |I_0| 
=[w]\ci{\infty, \tmop{cl}} \langle w\rangle\ci{I_0} ,
\]
so the theorem is proved modulo the main inequality \eqref{eq:main Bellman inequality}.

To proof \eqref{eq:main Bellman inequality} let us decompose the Haar function $h\ci I$ into the sum of elementary Haar functions $h\ci{I,k}$, $h=\sum_k h\ci{I,k}$, see Lemma \ref{l: elementary Haar decomposition}. 

It follows from \eqref{eq: sum |h_k|} that 
\begin{align}
\label{eq:norm h_1 le norm h}
 \| h\ci{I, k}
   \|\ci{L^2 (w)} \le \| h\ci I \|\ci{L^2 (w)} . 
\end{align}
Certainly
\[ 
(w, h\ci I)\ci {L^2} = \sum^n_{k = 1} (w, h\ci {I, k})\ci {L^2}, 
\]
so there exists a $k$ so that
\begin{align}
\label{eq: (w,h_k)}
| (w, h\ci {I, k})\ci {L^2} | \geqslant \frac{1}{n} | (w, h\ci I)\ci {L^2} | . 
\end{align}
Without loss of generality (by rearranging the intervals, if necessary) we can assume that this $k = 1$ and that the elementary Haar function $h\ci {I, 1}$ is a dyadic Haar function supported on the
first two $n$-adic subintervals $I_1$ and $I_2$ of $I$. 

Denote $I^1=I_1\cup I_2$. Then
\begin{align*}
N\ci I = \frac{2}{n} N\ci{I^1} + \frac1n \sum_{k=1}^n N\ci{I_k}, \qquad \text{and}\qquad N\ci{I^1} = \frac12 \left( N\ci{I_1} +N\ci{I_2} \right)
\end{align*}
By concavity of $B$ we get
\[ 
| I | B (N\ci {I}) \geqslant \sum^n_{k = 3} \frac{| I|}{n} B (N_k) +
   \frac{2}{n} | I | B \left( \frac{N_1 + N_2}{2} \right) . 
 \]
 
Note that for the elementary Haar function $h\ci{I,1}$
\begin{align*}
\frac{(w, h\ci{I,
   1})\ci{L^2}^2}{\| h\ci{I, 1} \|^2\ci{L^2 (w)}} = \frac{\left( \langle w\rangle\ci{I_1} - 
   \langle w\rangle\ci{I_2} \right)^2}{\langle w\rangle\ci{I^1} }|I^1| = 
   4\frac{\left( \langle w\rangle\ci{I_1} - 
   \langle w\rangle\ci{I^1} \right)^2}{\langle w\rangle\ci{I^1} } |I^1|
\end{align*}
Then applying Lemma \ref{l:two term Bellman inequality} and noticing that $\sd u$ in \eqref{eq:two term Bellman inequality} us exactly $\langle w\rangle\ci{I_1} - 
   \langle w\rangle\ci{I^1}$ we get 
\begin{align*}
| {I^1} | B \left( \frac{N_1 + N_2}{2} \right) & \geqslant 
\frac{|
  {I^1} |}{2} (B (N_1) + B (N_2)) + 2 \frac{(w, h\ci{I,
   1})\ci{L^2}^2}{\| h\ci{I, 1} \|^2_{L^2 (w)}}  &\qquad & \text{by \eqref{eq:two term Bellman inequality}}
\\
& \geqslant \frac{| {I^1}
   |}{2} (B (N_1) + B (N_2)) + \frac{2}{n^2} \frac{(w, h_{I})_{L^2}^2}{\|
   h_{I} \|^2_{L^2 (w)}}  && \text{by \eqref{eq:norm h_1 le norm h} and \eqref{eq: (w,h_k)}}. 
\end{align*}
The main inequality \eqref{eq:main Bellman inequality}, and so the theorem is proved.

\subsection{Some remarks}

It is a remarkable result of \cite{ChWiWo1985} that for any $Q \subset \mathbb{R}^{n}$ we have superexponential bound 
\begin{align}\label{distb}
\frac{1}{|Q|}\left|  \left \{ x  \in Q\; : \; f(x) - \langle f \rangle_{Q} \geq  \lambda  \right\} \right| \leq  e^{-\lambda^{2}/(2 \|{S}_\infty f\|_{\infty}^{2})}
\end{align}
for any $\lambda \geq 0$ and any $f$ with $\|{S}_\infty f\|_{\infty}<\infty$, where the square function ${S}_\infty$ is defined as follows 
\begin{align*}
{S}_\infty f = \left( \sum_{I \in \cD(Q)} \|\Delta\ci{I} f \|_{\infty}^{2} \1\ci{I}\right)^{1/2}. 
\end{align*}
The superexponential estimate allowed Wilson~\cite{Wilson1989-2} 
to obtain weighted $L^{p}$ estimates for the square function in terms of the maximal function, namely for any $0<p<\infty$ we have 
\begin{align}\label{wilson}
\int |M\ci\cD f|^{p} w\dd x \underset{n,p}{\lesssim}  [w]^{p/2}_{\infty} \int ({S}_\infty f)^{p} w \dd x 
\end{align}

For the standard dyadic filtration $S_\infty$ coincides with our square function $S$, so the result of Wilson (for $p=2$) gives for the standard dyadic filtration the statement of Theorem \ref{thm:
Ainfinity n adic}. However, this approach does not give  Theorem \ref{thm:
Ainfinity n adic} for $n$-adic filtration with $n\ge 3$, because the superexponential estimate  \eqref{distb} should be first proved for our square function $S$. And the square function $S_\infty$ is significantly larger than $S$: one can easily construct an example of a function with  $\|Sf\|_\infty\le 1$ and unbounded $S_\infty f$. So Theorem \ref{thm: Ainfinity n adic} is a new result. 

We should mention that it is possible using some ideas from the proof of Theorem \ref{thm:
Ainfinity n adic} to prove the estimate \eqref{distb} for our square function $S$. The reasoning from \cite{Wilson1989-2} then allows us to get the estimate \eqref{wilson} for our square function, but this will be a subject of a separate paper.

\section{Upper bound for the square function}
\label{S: upper bound for the square function}

In this section we sketch a proof of the harder estimate \eqref{eq: upper bound harder} in Theorem \ref{t:UpperBd}; the easier estimate \eqref{eq: upper bound easy} was proved earlier in Section \ref{s: trivial estimates}.

Trivial reasoning shows that it is sufficient to prove the estimate for an atomic filtration on $I_0=[0,1]$.  

The proof is based on the sparse domination of the square function. 

Recall that a collection $\cS\subset\cD$ is called \emph{sparse} if for any $J\in\cS$
\begin{align*}
\sum_{I\in \ch\ci\cS J} |I| \le |J|/2. 
\end{align*}
Given a sparse family $\cS$ the \emph{sparse square function} $S\ci\cS$ is defined as 
\begin{align*}
S\ci\cS f(x) := \biggl( \sum_{I\in \cS} \langle |f| \rangle^2\ci I \1\ci I(x) \biggr)^{1/2}
\end{align*}
\begin{lm}
\label{l: sparse square function}
Let $f\in L^1(I_0)$. There exist a sparse collection $\cS\subset \cD$ (depending on $f$) such that 
\begin{align*}
Sf(x) \lesssim S\ci \cS f(x) \qquad \text{a.e.}
\end{align*}
\end{lm}
\begin{proof}
The construction is pretty standard, we just outline it.

It is well known that the operator $S$ has weak type $1$-$1$, see \cite{Burk_MartTr1966}. The maximal function $M^\cD$ also has weak type $1$-$1$,  so there exists constant $C$ such that 
\begin{align}
\label{eq: square function weak type}
\left| \{x\in J: S\ci J f(x) >C \}\bigcup \{x\in J: M\ci J^\cD f(x) >C \}\right| \le |J|/2;
\end{align}
here $S\ci J$ is the \emph{localized} square function 
\begin{align*}
S\ci J f (x) := \biggl( \sum_{I\in\cD(J)} |\Delta\ci I f(x)|^2 \biggr)^{1/2}. 
\end{align*}

We start from the interval $I_0$. We define the stopping intervals $I\in \cS_1(I_0)$ to be the maximal (by inclusion) intervals $I\in\cD(I_0)$ such that either 
\begin{align*}
\langle |f| \rangle\ci I > C \langle |f|\rangle\ci{I_0} \qquad \text{or} \qquad \sum_{J\in\cD(I_0): I\subsetneqq J} |\Delta\ci J f (x)|^2 > C^2 \langle |f|\rangle\ci{I_0}^2\,;
\end{align*}
here $C$ is from \eqref{eq: square function weak type} and clearly $S=S\ci{I_0}$. 

By \eqref{eq: square function weak type} we have $\sum_{I\in\cS_1(I_0)}|I|\le |I_0|/2$, and 
\begin{align*}
Sf(x)^2 \le 3 C^2  \langle |f|\rangle\ci{I_0}^2 \1\ci{I_0} + 2C^2\sum_{I\in\cS_1(I_0)} \langle |f|\rangle\ci I \1\ci I + \sum_{I\in\cS_1(I_0)} S\ci I f(x)^2.
\end{align*}
Repeating this procedure for stopping intervals $I\in\cS_1(I_0)$ and iterating, we get the conclusion of the lemma.  
\end{proof}

\begin{proof}[Proof of estimate \eqref{eq: upper bound harder}]
It is sufficient to show that for a sparse family $\cS$
\begin{align*}
\|S\ci\cS f\|\ci{L^2(w)} \lesssim [w]\ci{2,\cD}^{1/2} [w^{-1}]\ci{\infty,\cD}^{1/2} \|f\|\ci{L^2(w)}
\end{align*}
Denoting $g=wf$, so $f=w^{-1} g$ we can rewrite this estimate as 
\begin{align}
\label{eq: sparse sq function estimate 02}
\|S\ci\cS (gw^{-1})\|\ci{L^2(w)} \lesssim [w]\ci{2,\cD}^{1/2} [w^{-1}]\ci{\infty,\cD}^{1/2} \|g\|\ci{L^2(w^{-1})}
\end{align}
So, we need to estimate 
\begin{align}
\label{eq: sparse sq function estimate 03}
\sum_{I\in\cS} \langle |g|w^{-1}\rangle\ci I^2 \langle w \rangle\ci I |I|
\end{align}
(the left hand side in \eqref{eq: sparse sq function estimate 02} squared). But as we already discussed above in Section \ref{s:better lower bound}, the martingale Carleson Embedding Theorem implies that it is sufficient to estimate \eqref{eq: sparse sq function estimate 03} on  functions $g=\1\ci J$, $J\in\cD$. Namely, if for all $J\in\cD$
\begin{align*}
\sum_{I\in\cS:I\subset J} \langle w^{-1}\rangle\ci I^2 \langle w \rangle\ci I |I| \le C \langle w^{-1}\rangle\ci J |J|
\end{align*}
then for all $g\in L^2(w^{-1})$, the sum \eqref{eq: sparse sq function estimate 03} is bounded by $4C \|g\|\ci{L^2(w^{-1})}^2$. 

Estimating we get
\begin{align*}
\sum_{I\in\cS:I\subset J} \langle w^{-1}\rangle\ci I^2 \langle w \rangle\ci I |I| & \le 
[w]\ci{2,\cD} \sum_{I\in\cS:I\subset J}  \langle w^{-1}\rangle\ci I |I|
\\
& \le  [w]\ci{2,\cD}  \| M\ci J (w^{-1}) \|\ci{L^1} 
\\
& \le   [w]\ci{2,\cD}  [w^{-1}]\ci{\infty,\cD}.
\end{align*}

\end{proof}

\def\cprime{$'$}
  \def\lfhook#1{\setbox0=\hbox{#1}{\ooalign{\hidewidth\lower1.5ex\hbox{'}\hidewidth\crcr\unhbox0}}}
\providecommand{\bysame}{\leavevmode\hbox to3em{\hrulefill}\thinspace}
\providecommand{\MR}{\relax\ifhmode\unskip\space\fi MR }
\providecommand{\MRhref}[2]{%
  \href{http://www.ams.org/mathscinet-getitem?mr=#1}{#2}
}
\providecommand{\href}[2]{#2}

\end{document}